\documentclass[11 pt, onecolumn]{ieeeconf}  
\usepackage[top=2cm, bottom=2cm, inner=2.5cm, outer=2.5cm]{geometry}

\IEEEoverridecommandlockouts                              

\usepackage{amsmath,amsfonts,amssymb}
\usepackage{mathabx,fixmath,algorithm, algpseudocode}
\usepackage{graphicx,color} 
\usepackage{lipsum}
\usepackage{cite}
\usepackage{hyperref}
\usepackage{multirow, booktabs}
\usepackage{pgfplots}

\newcommand{\qed}{\hfill $\diamond$}
\newcommand{\reals}{\mathbb{R}}

\newcommand{\Ez}{\underset{z \sim \mathcal{D}_{\mathrm{ms}}(y_k)}{\mathbb{E}}}
\newcommand{\Exi}{\underset{z \sim \mathcal{D}(y_k)}{\mathbb{E}}}
\newcommand{\Ey}{\underset{z \sim \mathcal{D}(y)}{\mathbb{E}}}
\newcommand{\Ew}{\underset{\omega \sim \mathcal{D}(y)}{\mathbb{E}}}

\newcommand{\A}{\mathbold{A}}
\newcommand{\bv}{\mathbold{b}}
\newcommand{\Prob}{\mathbb{P}}

\usepackage{palatino}

\newtheorem{theorem}{Theorem}
\newtheorem{example}{Example}

\newtheorem{assumption}{Assumption}

\title{\LARGE \bf \sf Flexible Optimization for Cyber-Physical and Human Systems}

\author{
Andrea Simonetto$^{*}$ \\ \smallskip
\small Unité de Mathématiques Appliquées, ENSTA Paris, Institut Polytechnique de Paris,\\ 91120 Palaiseau, France\\
\tt\small andrea.simonetto@ensta-paris.fr
\thanks{*This work was partly supported by the Agence Nationale de la Recherche (ANR) with the project ANR-23-CE48-0011-01.}
}

 \renewcommand{\baselinestretch}{1.1} 
 
\begin{document}

\maketitle
\thispagestyle{empty}

\begin{abstract}
Can we allow humans to pick among different, yet reasonably similar, decisions? Are we able to construct optimization problems whose outcome are sets of feasible, close-to-optimal decisions for human users to pick from, instead of a single, hardly explainable, do-as-I-say ``optimal'' directive? 

In this paper, we explore two complementary ways to render optimization problems stemming from cyber-physical applications flexible. In doing so, the optimization outcome is a trade off between engineering best and flexibility for the users to decide to do something slightly different. The first method is based on robust optimization and convex reformulations. The second method is stochastic and inspired from stochastic optimization with decision-dependent distributions. 
\end{abstract}



\section{INTRODUCTION}

Modern cyber-physical systems, such as the smart energy grid, are becoming tightly interlocked with the end users. Optimizing the operations of such systems is also being driven to the limits, by proposing personalized solutions to each user, e.g., to regulate their energy consumption. Sometimes we refer to these highly integrated systems as cyber-physical and human systems (CPHS)~\cite[Chapter~4D]{annaswamy2023control}. 

In this paper, we ask the question of whether we can optimize these systems by still allowing the end users to have a choice between different, yet reasonably similar, decisions. This becomes key in unlocking flexibility of the optimized decision to account for transparency and ease technology adoption. To fix the ideas on a concrete example, we could refer to optimizing a building heating control, where the set temperatures are determined by an algorithm. In this paper then, we study how to build algorithms that can deliver an allowed range of potentially good temperatures to the users to choose from independently. 

On the one hand, human behavior and satisfaction modeling is a well-studied research area, and therefore optimizing a cyber-physical system with pre-trained or online-learned human models has received much attention (see, e.g.,~\cite{Chatupromwong2012, Pinsler2018, Bourgin2019, Simonetto2021, zheng2021curve, Sadowska2023, annaswamy2023control} and references therein). On the other hand, unlocking flexibility by delivering sets and \emph{not single optimal solutions} to the users to choose from is not well explored, and mostly novel in optimization. In this paper, we were mainly inspired from the pioneering works~\cite{Inoue2019, Shibasaki2020} which propose a set-delivering controller. Their analysis techniques stems from robust control and inverse optimization, which we will not use here since our setting is different. 

In this paper, we propose the following main contributions,

{\bf[1]} We propose a deterministic flexible optimization problem that can deliver to end users a set of feasible solutions to pick from. This first contribution is rooted in robust optimization and it is made general by the latest techniques~\cite{bertsimas2022robust} to derive convex reformulations;

{\bf[2]} we propose a stochastic variant of the flexible optimization problem, which is less conservative and can fine tune the human-machine interaction. To solve this problem, we propose two primal-dual methods and prove their theoretical properties. The analysis of these algorithms is made possible thanks to recent developments in stochastic optimization with decision-dependent distributions~\cite{perdomo2020performative,drusvyatskiy2023stochastic, wood2023stochastic,lin2023plug,wang2023constrained}. 

The contributions yield two complementary views in flexible optimization and we finish by proposing a complete workflow, labeled Flex-O. 

Numerical experiments showcase our theoretical development and their empirical performance. 

\section{Problem formulation}

Let $f: \reals^n \to \reals$ be a convex cost, and let $X \subseteq \reals^n$ be the convex feasible set. We model the problem we want to solve as a convex optimization problem,
\begin{equation}\label{eq:standard}
\textsf{P}_1:\quad \min_{x\in \reals^n} f(x), \qquad \textrm{subject to } x \in X,
\end{equation}
where we partition the decision variable $x = [x_1 \in \reals^{n_1}, \ldots, x_N \in \reals^{n_N}] \in \reals^n$ to highlight the presence of $N$ users. For the sake of simplicity, we will let $n_i=1$ without an over loss of generality.  

Problem~\eqref{eq:standard} is rather standard and a variety of methods exist to find the optimal decisions. Here however, we wish to modify it to allow the users to have a choice. As explained in the introduction, we would like to assign to each user, not a single decision, but a set from which they can choose from. 
We present two ways that can be used to achieve this. 

\subsection{Deterministic approach}

We start by looking at a deterministic approach. The intuition is to find the best decision $x^\star$ and an hyperbox of optimal size centered on it, so that all the points in the hyperbox are in the feasible set. This will allow us to assign to each user their component of $x^\star$ and the possible variations around it, determined by the size of the hyperbox. 

To fix the ideas, Figure~\ref{fig.1} depicts the intuition in a bi-dimensional setting. As we can see, depending on the nature of each user, we may allow for more or less flexibility. 

\begin{figure}
\centering
\includegraphics[width=0.50\textwidth]{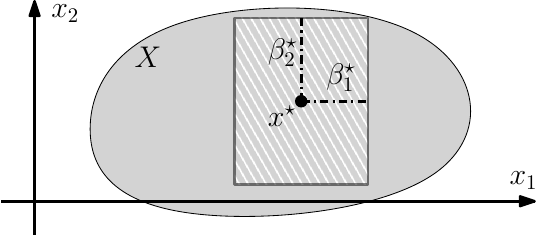}
\caption{Setting of the problem formulation in two dimensions with the best decision $x^\star$ and the optimal variations around it $\beta^\star$. We give to each user $i$ the optimal set $[x^\star_i - \beta^\star_i, x^\star_i + \beta^\star_i]$. }
\label{fig.1}
\end{figure}

We are now ready for formulating the problem mathematically. Let us introduce scalar weights $w_i\geq 0$ and new scalar flexibility variables $\beta_i \geq 0$ for each user. We also introduce an uncertain variable $z_i \in [-1,1]\subset \mathbb{R}$ for each user. We collect $\beta_i$ and $z_i$ in column vectors as $\beta = (\beta_1, \ldots, \beta_n)^\top; z = (z_1, \ldots, z_n)^\top \in [-1,1]^n$, where $[-1,1]^n$ is the unitary hyperbox. We also define $[{\bf diag}(z)]$ as a $n \times n$ matrix with $z$ as its diagonal. Then, we render $\textsf{P}_1$ flexible by solving instead the following robust optimization problem,
\begin{eqnarray}
\textsf{P}_{2,\mathrm{d}}: \,\, (x^\star, \beta^\star) \in &&\hspace*{-0.5cm}\arg\!\!\!\min_{\hspace*{-0.5cm}x \in \mathbb{R}^n, \beta_i \geq 0}\quad  f(x) - \sum_{i=1}^n w_i\beta_i  , \label{eq.p2}\\ &&\hspace*{-0.5cm}\textrm{subject to} \quad x+[{\bf diag}(z)] \beta \in X, \, \nonumber \\ &&\hspace*{3cm} \forall z \in [-1,1]^n. \nonumber
\end{eqnarray}
Problem $\textsf{P}_{2,\mathrm{d}}$ yields feasible solutions for all decisions in the hyperbox $X^\star:=\prod_{i=1}^n X^\star_i$, with $X^\star_i:=[x_i^\star - \beta_i^\star , x_i^\star + \beta_i^\star]$. The set $X^\star_i$ can be given to the user $i$ for them to decide their optimal action, \emph{independently of the other users}. Here, we are trading-off optimality of $x^\star$ while increasing the flexibility ensured by $\beta^\star$. That is, we are finding the optimal point $x^\star$ \emph{and} the maximal variation around it $\beta^\star$, which still guarantees feasibility. 

For the sake of generality, Problem~\eqref{eq.p2} can be slightly generalized into,
\begin{eqnarray}
\textsf{P}^{\varphi}_{2,\mathrm{d}}: \,\, (x^\star, \beta^\star) \in &&\hspace*{-0.5cm}\arg\!\!\!\min_{\hspace*{-0.5cm}x \in \mathbb{R}^n, \beta_i \geq 0}\quad  f(x) + \sum_{i=1}^n w_i \varphi_i(\beta_i)  , \label{eq.p2p}\\ &&\hspace*{-0.5cm}\textrm{subject to} \quad x+[{\bf diag}(z)] \beta \in X, \, \nonumber \\ &&\hspace*{3cm} \forall z \in [-1,1]^n, \nonumber
\end{eqnarray}
for any convex function $\varphi_i: \reals \to \reals$. A typical example would be $\varphi_i(\beta_i) = -\beta_i + \frac{\epsilon}{2} \beta_i^2$, furthering trading-off flexibility for the user (small $\epsilon$) and flexibility for the system designed (large $\epsilon$). 

Problems~\eqref{eq.p2} and \eqref{eq.p2p} are difficult optimization problems, for which however some approximation and reformulation procedures exist. We will discuss some of them in the following sections, and for that it is convenient to adapt slightly the notation. We let $y := [x^\top, \beta^\top]^\top$, $Y := \reals^n \times \reals^n_{+}$,  we introduce the new convex cost $g(y) := f(x) + \sum_{i=1}^n w_i\varphi_i(\beta_i) $, and rewrite $x+[{\bf diag}(z)] \beta \in X$ as the intersection of finitely many inequalities (say $m$), as
\begin{equation}
h(\A(y) z + \bv(y)) \leq 0,
\end{equation} 
for convex function $h:\reals^{n} \to \reals^m$ and affine in $y$ matrix $\A(y) = [{\bf diag}(\beta)]$ and vector $\bv(y) = x$. This reformulation is very often possible in all the applications we consider. Then the problem reads,
\begin{eqnarray}
\textsf{P}_{3,\mathrm{d}}: \,\, y^\star \in &&\hspace*{-0.5cm}\arg\!\min_{\hspace*{-0.5cm}y \in Y}\quad  g(y)  , \label{eq.p3}\\ &&\hspace*{-0.5cm}\textrm{subject to} \quad h(\A(y) z + \bv(y)) \leq 0, \, \nonumber \\ &&\hspace*{3cm} \forall z \in [-1,1]^n. \nonumber
\end{eqnarray}

\begin{example}\label{ex.1}
We consider the task of deciding the reference temperatures in different areas in an office building. Each group of users can set their thermostat in their office within an allowed range which we need to provide. Let $x \in \reals^n$ be the temperature in $n$ different areas, and $x_{\mathrm{ref}}\in \reals^n$ be the engineering-best temperatures, which have been determined via an economic welfare trade-off. The problem we would like to solve can be the following one,
\begin{eqnarray}
y^\star \in &&\hspace*{-0.5cm}\arg\!\min_{\hspace*{-0.5cm}y \in Y}\quad  \frac{1}{2}\|x\|^2 + \sum_{i=1}^n w_i \left(-\beta_i+ \frac{\epsilon}{2} \beta_i^2\right), \label{eq.ex}\\ &&\hspace*{-0.5cm}\textrm{subject to} \quad \left.\begin{array}{l}\|\A(y) z + \bv(y) - x_{\mathrm{ref}}\|^2 \leq \gamma,\\ D\, ( \A(y) z + \bv(y)) \leq e, \end{array}\right\} \nonumber \\ &&\hspace*{5cm} \forall z \in [-1,1]^n. \nonumber
\end{eqnarray}
Here, the cost represents the wish to pick the smallest possible temperature, while the constraints impose a limited deviation with $x_{\mathrm{ref}}$ via a nonnegative scalar $\gamma$, and some additional affine constraints $D \in \reals^{c \times n}, e \in \reals^{c}$. The latter ones impose additional temperature bounds, and the fact that close-by areas cannot have very different temperatures. \qed
\end{example}




\subsection{Stochastic approach}

In order to refine, and possibly render Problem~\eqref{eq.p3} less conservative, we introduce a stochastic variant. Here, we assume that the users, \emph{given a certain optimal decision $x^\star_i$ and allowed variation $\beta^\star_i$}, they pick a variable in the optimal set, say $x_i$, with a certain probability. 

In this case, by introducing a nonnegative scalar $\delta \in[0,1]$, Problem~\eqref{eq.p3} can be formulated as a chance-constrained decision-dependent non-convex problem as,
\begin{eqnarray}
\textsf{P}_{2,\mathrm{s}}: \,\, y^\star \in &&\hspace*{-0.5cm}\arg\!\min_{\hspace*{-0.5cm}y \in Y}\quad  g(y)  , \label{eq.s2}\\ &&\hspace*{-0.8cm}\textrm{subject to} \quad \Prob_{z \sim \mathcal{D}(y)}\left[h(\A(y) z + \bv(y)) > 0\right] \leq  \delta, \, \nonumber \end{eqnarray}
where $\Prob[\cdot]$ is the probability of a certain event, and $\mathcal{D}(y)$ is the decision-dependent distribution from which $z$ is drawn from, whose support we assume is $[-1,1]^n$.


The chance constraint in Problem~\eqref{eq.s2} can be conservatively rewritten in a convex-in-$h(y)$ form, by employing several standard bounds. Here, for reasons that will be clear in the algorithmic section, we need a smooth reformulation and we employ the Chernoff's bound. This allows one to write,
\begin{multline}
\Prob_{z \sim \mathcal{D}(y)}\left[h(\A(y) z + \bv(y)) > 0\right] \leq \delta \Longleftarrow \\ \mathbb{E}_{z \sim \mathcal{D}(y)}\left(\exp[h_j(\A(y) z + \bv(y))/u]\right) - \delta \leq 0, \quad \mathrm{for} \, j = 1, \ldots, m. \end{multline}
for any $u>0$; see for instance~\cite{nemirovski2007convex}, where the fact that the random variable $z$ is decision-dependent does not affect the bound reasoning.  With this in place, Problem~\eqref{eq.s2} can be reformulated as,
\begin{eqnarray}
\textsf{P}_{3,\mathrm{s}}: \,\, y^\star \in &&\hspace*{-0.5cm}\arg\!\min_{\hspace*{-0.5cm}y \in Y}\quad  g(y)  , \label{eq.s3}\\ &&\hspace*{-0.5cm}\textrm{subject to} \nonumber \\ &&\hspace*{-1.0cm}  \mathbb{E}_{z \sim \mathcal{D}(y)}\left(\exp[h_j(\A(y) z + \bv(y))/u]\right) - \delta \leq 0, \quad \mathrm{for} \, j = 1, \ldots, m. \, \nonumber 
\end{eqnarray}
Note that, contrary to the case of decision-independent distribution, the constraints in~\eqref{eq.s3} are still non-convex in $y$.  

Finally, we can write Problem~\eqref{eq.s3} by its minimax formulation,
\begin{equation}
\textsf{P}_{4,\mathrm{s}}: \,\, \min_{y \in \mathbb{R}^{2n}}\max_{\lambda \in \reals^{m}_{+}}\quad \Phi(y, \lambda):= g(y) +  \sum_{j} \lambda_j (\mathbb{E}_{z \sim \mathcal{D}(y)}\left(\exp[h_j(\A(y) z + \bv(y))/u]\right) - \delta) \label{minmax}
\end{equation}
where the variables $\lambda \in \reals^{m}_+$ are the Lagrangian multipliers associated to the constraints.

\section{SOLVING THE FLEXIBLE PROBLEM}

\subsection{Robust optimization problem}\label{sec.ro}

Several techniques exist to approximately (and conservatively) solve robust optimization problems\footnote{A naive approach would be to verify the constraints on all the vertices of the hybercube $[-1,1]^n$, but that would lead to $m 2^n$ constraints.} like~\eqref{eq.p3}. A recent framework, based on an extension of the Reformulation-Linearization-Technique, has been proposed by Bertsimas and coauthors in~\cite{bertsimas2022robust}. This framework is able to deal with any scalar convex function $h(\A(y)z + \bv(y)) \leq 0$ and any uncertainty set, and therefore it can be used here. 

For the sake of argument, we will not discuss this latest wholistic approach for the general case, but look at more standard techniques for our Example~\ref{ex.1}. E.g., one can transform Problem~\eqref{eq.ex} into the convex worst-case reformulation,
\begin{eqnarray}
(y^\star, s^\star) \in &&\hspace*{-0.5cm}\arg\!\!\min_{\hspace*{-0.5cm}y \in Y, s\in\reals^n}\quad  \frac{1}{2}\|x\|^2 + \sum_{i=1}^n w_i \left(-\beta_i+ \frac{\epsilon}{2} \beta_i^2\right), \label{eq.ex-rob}\\ &&\hspace*{-1.5cm}\textrm{subject to} \quad \|s\|^2 \leq \gamma, \,  \nonumber \\ &&\hspace*{0.25cm} s_i \geq |\beta_i| + |x_i - x_{\mathrm{ref},i}|, \, \forall i \in \{1, n\}\nonumber  \\ &&\hspace*{0.15cm} d_j x - e_j +  \|[{\bf diag}(d_j)] \beta\|_1 \leq 0,\, \forall j \in \{1, c\} \nonumber
\end{eqnarray}
see for instance~\cite{Boyd2004a, ben2009robust, duchi2018optimization} and Appendix~\ref{ap.1} for completeness, where $d_j$ are the rows of $D$ and $e_j$ are the components of $e$. 

Other techniques, such as the S-procedure, can be applied to specific cases and the reader is referred to~\cite{ben2009robust}. For the sake of this paper, we remark that finding convex reformulations to robust problems like~\eqref{eq.p3} is possible, even if somewhat conservative.  
The resolution of such reformulations, like~\eqref{eq.ex-rob}, yields the optimal decision $x^\star$, as well as the optimal interval around it $\beta^\star$. As this may be conservative, we turn to the stochastic approach to refine it.

\subsection{Stochastic optimization problem}

We start our resolution strategy by rewriting~\eqref{minmax} in the compact form,
\begin{equation}\label{st.sp}
\min_{y \in Y}\max_{\lambda \in \reals^{m}_{+}}\quad \Phi(y, \lambda):= \mathbb{E}_{z \sim \mathcal{D}(y)} [\phi (y, \lambda, z)].
\end{equation}
Problem~\eqref{st.sp} is a stochastic saddle-point problem with decision-dependent distributions, which is in general non-convex and intractable in practice since one would need a full (local) characterization of $\mathcal{D}$. 
For this class of problems, since the optimizers are out of reach, one is content to find equilibrium points, as the points that are optimal w.r.t. the distribution they induce. In particular, one would start by assuming that the search space in $y$ and $\lambda$ is compact. In our case, this would be a reasonable approximation for $y$, since we can get an educated guess of a bounded search space by solving the deterministic problem~\eqref{eq.p3} first. For $\lambda$, that would amount at clipping the multipliers, which is also a reasonable practice in convex and non-convex problems \cite{Nedic2009a, Erseghe2015}. With this in place, we let the search space for $y$ be $\mathcal{Y} \subset Y$ and $\lambda \in \mathcal{M} \subset \reals^{m}_{+}$. Then, one searches for equilibrium points, such that,
\begin{eqnarray}
\bar{y} &\in &\arg\min_{y \in \mathcal{Y}} \left\{ \max_{\lambda \in \mathcal{M}} \mathbb{E}_{z \sim \mathcal{D}(\bar{y})} [\phi (y, \lambda, z)] \right\}\\
\bar{\lambda} &\in &\arg\max_{\lambda \in \mathcal{M}} \left\{ \min_{y \in \mathcal{Y}}  \mathbb{E}_{z \sim \mathcal{D}(\bar{y})} [\phi (y, \lambda, z)] \right\}.
\end{eqnarray}

For continuous convex (in $y$) concave (in $\lambda)$ uniformly in $z$ functions $\phi$, as in our case, assuming compactness of the sets $\mathcal{Y}$ and $\mathcal{M}$, as well as a continuous distributional map $\mathcal{D}$ under Wasserstein-$1$ distance $W_1$, then we know that the set of equilibrium points is nonempty and compact~\cite[Thm~2.5]{wood2023stochastic}. Consider further the following requirements. 

\begin{assumption}\label{as.1}
\emph{(a)} Function $\phi$ is continuously differentiable over $Y \times \reals^{m}_{+}$ uniformly in $z$, as well as $\mu$-strongly-convex-strongly-concave (respectively in $y$ and $\lambda$) for all $z$. 

\emph{(b)} The stochastic gradient map $\psi(y,\lambda,z) := (\nabla_{y} \phi(y,\lambda,z), -\nabla_{\lambda} \phi(y,\lambda,z))$ is jointly $L$-Lipschitz in $(y, \lambda)$ and separately in $z$.  

\emph{(c)} The distribution map $\mathcal{D}$ is $\varepsilon$-Lipschitz with respect to the Wasserstein-$1$ distance $W_1$, i.e., 
$$
W_1(\mathcal{D}(y), \mathcal{D}(y')) \leq \varepsilon \|y-y'\|, \quad \forall y \in Y. 
\vspace*{-3mm}$$ \qed
\end{assumption}

Strong convexity is ensured by properly defining the cost function, as well as $\epsilon$, while differentiability holds thanks to the use of Chernoff's bound. If the engineering function is just convex, a regularization may be added. Strong concavity can be achieved by adding the dual regularization term $-\frac{\nu}{2}\|\lambda\|^2$, $\nu \geq \mu$, as in~\cite{Nedic2011}. 
The various Lipschitz assumptions are mild (the ones on the gradient map hold trivially under compactness of the sets $\mathcal{Y}, \mathcal{M}$ and compact support for $\mathcal{D}$ as assumed). 

With these assumptions in place, one can show that the distance between equilibrium points and optimal points of the original problem~\eqref{st.sp} is upper bounded by the constant of the problem and therefore solving for the former is a proxy for finding good approximate saddle points for the latter. 
Furthermore for $\varepsilon L /\mu <1$, then the equilibrium point is unique~\cite[Thm~2.10]{wood2023stochastic}. To find such unique equilibrium point $(\bar{y}, \bar{\lambda})$ we employ a stochastic primal-dual method by generating a sequence of points $\{{y}_k, {\lambda}_k\}$, $k = 0, 1, \ldots,$ as,
\begin{subequations}\label{pd}
\begin{eqnarray}
z_k &\sim& \mathcal{D}(y_k), \label{human} \\
y_{k+1} &= &{\mathsf P}_{\mathcal{Y}} \left[ y_k - \eta \nabla_{y}\phi (y_k, \lambda_k, z_k) \right], \\
\lambda_{k+1} &= &{\mathsf P}_{\mathcal{M}} \left[ \lambda_k + \eta \nabla_{\lambda}\phi (y_k, \lambda_k, z_k) \right],
\end{eqnarray}
\end{subequations}
with step size $\eta >0$ and projection operator ${\mathsf P}[\cdot]$. The stochastic gradient obtained by drawing $z_k$ from the distribution generated at $y_k$ is unbiased at $k$. Furthermore, we assume (as usual in the stochastic setting) that, for any given $y, \lambda$:
\begin{subequations}\label{st.setting}
\begin{equation}
\Ew[\|\nabla_{y}\phi (y, \lambda, \omega) - \Ey[\nabla_{y}\phi (y, \lambda, z)]\| \leq \frac{\sigma}{\sqrt{2}}, 
\end{equation}
\begin{equation}
\Ew[\|\nabla_{\lambda}\phi (y, \lambda, \omega) - \Ey[\nabla_{\lambda}\phi (y, \lambda, z)]\| \leq \frac{\sigma}{\sqrt{2}}, 
\end{equation}
\end{subequations}
for a nonnegative constant $\sigma$. 

Then we can derive the following result.
\begin{theorem}\label{th.1}
Let $p = [y^\top, \lambda^\top]^\top$. Let Assumption~\ref{as.1} and the stochastic setting~\eqref{st.setting} hold. Assume $ \frac{\varepsilon L}{\mu} <1$ and pick the step size $\eta$ as 
$$
\eta \in \left(0, \frac{2 (\mu - \varepsilon L)}{L^2 (1 - \varepsilon^2)}\right).
$$
Then the primal-dual method in Eq.~\eqref{pd} generates a sequence of points $\{{y}_k, {\lambda}_k\}$, such that in total expectation,
$$
\limsup_{k\to \infty} \mathbb{E}[\|p_k - \bar{p}\|] = \frac{\eta \sigma}{1 - \varrho}, 
$$
with $\varrho := \sqrt{1 - 2 \eta \mu + \eta^2 L^2} + \eta \varepsilon L <1$. \qed
\end{theorem}

\begin{proof}
It follows from~\cite{wood2023stochastic} and it is reported in Appendix~\ref{ap.2}. 
\end{proof}

Theorem~\ref{th.1} tells us that if the step size is chosen sufficiently small, we can generate a sequence of points that approximates the unique equilibrium up to an error ball. The size of this ball depends on the variance of the stochastic gradient, as in many stochastic settings. A more refined characterization of the primal-dual algorithm in terms of tail distributions can be found in~\cite{wood2023stochastic}, but it is qualitatively identical. We also remark that stochastic versions of optimistic gradient descent-ascent methods~\cite{malitsky2020forward,jiang2022generalized} might be employed here, even though their convergence in decision-dependent settings is still a non-trivial open problem.  

\section{A HUMAN-ADAPTED ALGORITHM}



Both solving the robust problem~\eqref{sec.ro} and the stochastic decision-dependent variant with~\eqref{pd} have their advantages and drawbacks. The robust program offers hard guarantee on feasibility but may be conservative. The primal-dual method can be closer to reality, however it achieves feasibility only asymptotically and it requires humans to ``play'' at each iteration (since \eqref{human} is achieved by asking humans to select their $z_k$), which can be unreasonable from an user-oriented perspective (e.g., if you are asked to adjust your thermostat multiple times). 

In this section, we present a middle ground. The idea is to approximate the user's choice~\eqref{human} by a model, which we can run without asking the users to ``play''. Then, a pertinent notion of convergence will be provided in terms of the miss-match between the chosen model and the real distribution.  

We chose to model users as intelligent agents who respond to the requests by employing a best-response mechanism. This concept has been studied in economics and in game theory~\cite{Kahneman1979,lin2023plug,wood2023solving} as well as in optimization~\cite{mohajerin2018data} and control~\cite{Shibasaki2020}. The idea is to model variable $z_i$ as if it was derived from a user-dependent optimization problem: humans want to select a $z_i$ which minimizes their discomfort. Examples of such models are additive shift rules such as,
\begin{equation}\label{br-m}
    z_i \overset{d}{=} \left[\, \Psi_i(x_i,\beta_i) + \xi_{i}\right]_{-1}^{+1}, \qquad \xi_i \sim \mathcal{D}_i,
\end{equation}
where $\mathcal{D}_i$ is a static distribution, $\overset{d}{=} $ indicates equality in distribution, and $[\cdot]_{-1}^{+1}$ indicates that the distribution is then truncated to have a $[-1,1]$ support. The reasoning behind \eqref{br-m} is that the users are selecting their best $z_i$ depending on $x_i, \beta_i$, plus a static noise. The function $\Psi_i(x_i,\beta_i): \reals^{2} \to \reals$ can be thought of as a function that encodes an optimization problem parametrized by $(x_i,\beta_i)$. 


Let $\mathcal{D}_{\mathrm{ms}}(y)$ be the decision-dependent distribution induced by considering model~\eqref{br-m}. We assume that the estimated model~\eqref{br-m} is misspecified up to an error $B>0$ as follows. 

\begin{assumption}\label{as.2}{Cf. \cite{lin2023plug,wood2023solving}.}
The distribution $\mathcal{D}_{\mathrm{ms}}(y)$ is $B$-misspecified, in the sense that there exists a nonnegative scalar $B$, such that
$$
W_1(\mathcal{D}_{\mathrm{ms}}(y), \mathcal{D}(y)) \leq B, \qquad \forall y \in \mathcal{Y}. 
\vspace*{-3mm}$$ \qed
\end{assumption}

Trivial bounds for $B$ can be easily derived in our setting, since $z \in [-1,1]^n$ for both true and estimated distributions, but better bounds can also be obtained with enough data on the users. 

With this in place, we then use the misspecified model to run a deterministic model-based primal-dual method as,
\begin{subequations}\label{pd-dm}
\begin{eqnarray}
y_{k+1} = {\mathsf P}_{\mathcal{Y}} \left[ y_k - \eta \Ez \nabla_{y}\phi (y_k, \lambda_k, z) \right], \\
\lambda_{k+1} = {\mathsf P}_{\mathcal{M}} \left[ \lambda_k + \eta \Ez \nabla_{\lambda}\phi (y_k, \lambda_k, z) \right].
\end{eqnarray}
\end{subequations}

The advantage of Eq.s~\eqref{pd-dm} is that they can be run \emph{without} human intervention. Naturally, they can also be extended to a mini-batch and stochastic mode, but we do not do that here. For model~\eqref{br-m}, we need Assumption~\ref{as.1} to hold, which requires that the model is Lipschitz with respect to $y$ as,
\begin{equation}
W_1(\mathcal{D}(y), \mathcal{D}(y'))  = \|\Psi(y) - \Psi(y') \| \leq \varepsilon \|y-y'\|, \quad \forall y \in Y. 
\end{equation}

For iterations~\eqref{pd-dm}, we have the following result.

\begin{theorem}\label{th.2}
Let $p = [y^\top, \lambda^\top]^\top$. Let Assumption~\ref{as.1} and the misspecified setting of Assumption~\ref{as.2} hold. Assume $ \frac{\varepsilon L}{\mu} <1$ and pick the step size $\eta$ as 
$$
\eta \in \left(0, \frac{2 (\mu - \varepsilon L)}{L^2 (1 - \varepsilon^2)}\right).
$$
Then the model-based primal-dual method in Eq.~\eqref{pd-dm} generates a sequence of points $\{{y}_k, {\lambda}_k\}$, such that deterministically,
$$
\limsup_{k\to \infty} \|p_k - \bar{p}\| = \frac{\sqrt{2} \eta L B}{1 - \varrho}, 
$$
with $\varrho := \sqrt{1 - 2 \eta \mu + \eta^2 L^2} + \eta \varepsilon L <1$. \qed    
\end{theorem}

\begin{proof}
It follows from~\cite{wood2023stochastic}, in particular by using Kantorovich and Rubinstein duality for the $W_1$ metric, as well as the Lipschitz assumption on the gradient (Cf. Assumption~\ref{as.1}-(b)), and it is reported in Appendix~\ref{ap.3}.
\end{proof}

Theorem~\ref{th.2} says the iterations~\eqref{pd-dm} converge up to an error bound, whose size is determined by the misspecification. If one used stochastic variants to estimate the average values in~\eqref{pd-dm}, one would also be able to derive bounds in expectation with extra error terms. 

\subsection{Warm-start, guarding, and rounding}

Considering~\eqref{pd-dm}, we can now describe a final human-adapted algorithm. Consider the following workflow:
\bigskip
\small
\hrule
\smallskip
{\bf FleX-O: Flexible optimization algorithm}
\smallskip
\hrule
\smallskip
\begin{enumerate}
\item Solve the robust optimization problem~\eqref{eq.p3} with the techniques of Section~\ref{sec.ro};
\item Use the solution from (1) to warm-start the iterations~\eqref{pd-dm} with an estimated model for $T$ iterations;
\item (Optional: {\bf Guarding step}) to make sure the solution is feasible $\forall z \in [-1,1]^n$, project the solution obtained from (2), say $y_T$ onto the robust feasible set, i.e., solve 
\begin{eqnarray}\label{robustproject}
\textsf{P}_{4,\mathrm{d}}: \,\, y^\star \in &&\hspace*{-0.5cm}\arg\!\min_{\hspace*{-0.5cm}y \in Y}\quad  \frac{1}{2}\|y - y_T\|^2  , \label{eq.p4}\\ &&\hspace*{-0.5cm}\textrm{subject to} \quad h(\A(y) z + \bv(y)) \leq 0, \, \nonumber \\ &&\hspace*{3cm} \forall z \in [-1,1]^n. \nonumber
\end{eqnarray}
with the techniques of Section~\ref{sec.ro};
\item (Optional: {\bf Rounding step}) to make the solution of (2) or (3) more human-friendly innerly round $\beta^\star$ to the closest precision human can achieve.
\item {\bf Output:} an optimal set $[x_i^\star - \beta_i^\star, x_i^\star + \beta^\star_i]$ for each user $i$.  
\end{enumerate}
\smallskip
\normalsize
\hrule



\section{NUMERICAL RESULTS}

To numerically illustrate the proposed approaches, we consider the setting of Example~\ref{ex.1}, where we set $n=7$. We further let the cost be:
\begin{equation}
   g(y) =  \frac{\epsilon_x}{2} \|x\|^2 + \sum_{i=1}^{7} w_i \left(-\beta_i + \frac{\epsilon_\beta}{2} \beta_i^2\right),  
\end{equation}
with $\epsilon_x = 0.001$, $\epsilon_\beta = 0.01$, the weights $w_i$ randomly drawn from the uniform distribution $\mathcal{U}(0.1,1)$. We also set $x_{\textrm{ref}}$ randomly from a normal distribution $\mathcal{N}(19.5, 1.)$. We remark that temperatures are expressed in degrees Celsius. We let $\gamma = 2n$. We consider a corridor with $n=7$ offices, and therefore $D$ is the matrix that represents the fact that two adjacent offices cannot have a very different temperature. In this case $c=6$, and we let $e = [1, \ldots, 1]^\top$. For the primal-dual methods we let the dual regularization be $\nu=0.01$, $u=1.5$, and $\delta = 0.2$. By trial-and-error, we fix the step size at $\eta=0.05$ for all the methods.  

In Table~\ref{tab.1}, we report the optimal solution of the problem obtained by the robust convex reformulation~\eqref{eq.ex-rob}. We see that the variations around the optimal ``imposed'' temperature are minimal in certain cases. We then use this solution as a warm start for the primal-dual methods. In Figure~\ref{fig.1}, we report the evolution of the primal distance $\|y_k - \bar{y}\|$ for the baseline primal-dual~\eqref{pd} {\bf [B-PD]}, for the misspecified-model-based primal-dual~\eqref{pd-dm} {\bf [MS-PD]}, and for Flex-O with a guarding step~\eqref{robustproject} at $T = 50, 500, 5000$, {\bf [Flex-O]}. In all the cases, a nearly optimal value $\bar{y}$ is computed as follows. First we model the true \emph{but unknown} distribution $\mathcal{D}(y)$ as,
\begin{equation}\label{truemodel}
z_i \overset{d}{=} \left[\xi_i + \left\{ \begin{array}{lr}
    \beta_i (19.0-x_i) & \textrm{if } x_i \leq 19.0 \\
    -\beta_i (x_i - 20.5) & \textrm{if } x_i \geq 20.5 \\
    0 & \textrm{otherwise} 
\end{array} \right. \right]_{-1}^{+1}, 
\end{equation}
with $\xi_i \sim \mathcal{N}(0,0.1)$. Then $\bar{y}$ is computed by running the model-based primal-dual~\eqref{pd-dm} on the true distribution~\eqref{truemodel}. For the misspecified model, we instead take the deterministic,
\begin{equation}
z_i = \max\{-1,\min\{1, -\beta_i (x_i - 19.75)\}\},
\end{equation}
which induces the distribution $\mathcal{D}_{\mathrm{ms}}(y)$. 

\begin{figure}
\centering
\resizebox{0.60\textwidth}{!}{\input{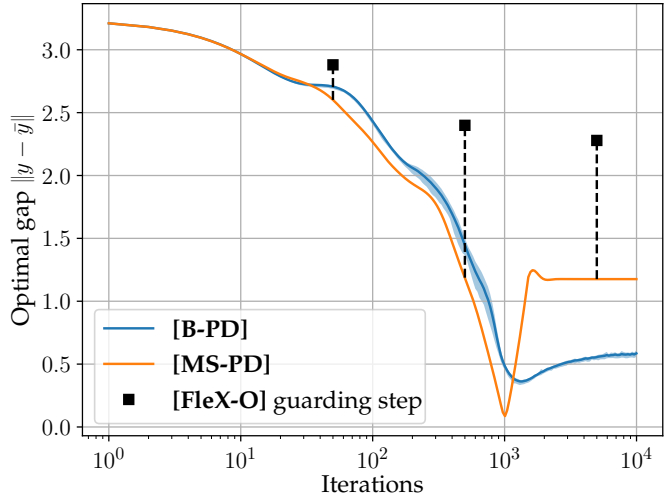}}
\caption{Optimality gap vs. iterations for the considered algorithms. For {\bf [B-PD]} we indicate mean and standard deviation over $50$ realizations. }
\label{fig.1}
\end{figure}

Model~\eqref{truemodel} indicates the natural propensity to accept the proposed optimal solution if it is within a suitable temperature range, while reacting if it falls outside. The strength of the reaction depends on the allowed variations. 

As we observe in Figure~\ref{fig.1}, both {\bf [B-PD]} and {\bf [MS-PD]} reduce the optimal gap, eventually reaching the error bound. For {\bf [B-PD]} we have averaged the solution over $50$ realizations. We also see the effect of the guarding step on {\bf [Flex-O]}, which makes the solution less optimal w.r.t. $\bar{y}$. 

More interestingly, if we observe the last iterate values in Table~\ref{tab.1}, we see how the proposed primal-dual algorithms offer more flexibility (i.e., higher values of $\beta$). In the table $\langle \mathrm{CV}(z) \rangle$ indicate the average value over the last $100$ iterations of the constraint violation,
$$
\max_{j} \{\mathbb{E}_{z \sim \mathcal{D}(y)}\left[ h_j(\A(y) z + \bv(y))\right]\}. 
$$
From the results, we can appreciate the importance of having a good model for optimality. 
For {\bf [Flex-O]}, we see how projecting onto the robust set trades-off flexibility with robustness. We note that even a few steps $\sim50$ of the primal-dual can unlock more flexible solutions.

\begin{table}
\renewcommand{\arraystretch}{1.1}
    \caption{Solutions of different algorithms on the test problem. }
    \label{tab.1}
    \centering
    \scalebox{0.9}{
\begin{tabular}{ccc}
\toprule Algorithm & Found $x$, found $\beta$ \\
\toprule Robust~\eqref{eq.ex-rob} & 
$\begin{array}{c} x^\star = [19.4, 19.4, 18.8, 18.3, 18.3, 18.3, 18.3] \\ 
\beta^\star = [1.0, 0.0, 0.5, 0.0, 1.0, 0.0, 1.0]\end{array}$ \\
\midrule True $\mathcal{D}(y)$ p.-d.~\eqref{pd-dm} & 
$\begin{array}{c} \bar{x} =  [17.7, 17.7, 17.7, 17.7, 17.7, 17.7, 17.7] \\
\bar{\beta} =  [1.0, 1.0, 1.0, 1.0, 1.0, 1.0, 1.0] \\ 
\langle \mathrm{CV}(z)\rangle = 0.023 \end{array}$\\ \midrule
{\bf [MS-PD]}~\eqref{pd-dm} &$\begin{array}{c} x =  [17.3, 17.3, 17.3, 17.3, 17.3, 17.3, 17.3] \\
\beta =  [1.0, 1.0, 1.0, 1.0, 1.0, 1.0, 1.0]\\ 
\langle \mathrm{CV}(z)\rangle =  -0.450\end{array}$\\\midrule
{\bf [B-PD]}~\eqref{pd} &$\begin{array}{c} x =  [18.0, 17.9, 17.9, 17.9, 17.9, 17.9, 17.9] \\
\beta =  [1.0, 0.9, 0.9, 0.9, 0.9, 0.9, 0.9]\\ 
\langle \mathrm{CV}(z)\rangle =  -0.018\end{array}$\\\midrule
{{\bf [Flex-O]}} at $T=50$ &
$\begin{array}{c} x =  [19.1, 19.1, 18.6, 18.3, 18.3, 18.4, 18.3]\\
\beta =  [0.7, 0.2, 0.3, 0.3, 0.6, 0.3, 0.5] \end{array}$ \\
{\color{white}{\bf [Flex-O]}} at $T=500$ &
$\begin{array}{c} x = [18.8, 18.5, 18.3, 18.1, 18.1, 18.4, 18.1]\\
\beta =  [0.3, 0.3, 0.3, 0.4, 0.5, 0.1, 0.5] \end{array}$ \\
{\color{white}{\bf [Flex-O]}} at $T=5000$ &
$\begin{array}{c} x = [18.5, 18.3, 17.9, 17.7, 17.7, 18.2, 17.8]\\
\beta =  [0.0, 0.0, 0.3, 0.4, 0.5, 0.0, 0.5] \end{array}$ \\
\bottomrule
\end{tabular}}
\end{table}

\section{CONCLUSIONS}

We have formulated flexible optimization problems that yield optimal decisions and per-user optimal variations around them. This allows the user to be given sets of possible decisions to take. The algorithms are based on robust optimization and stochastic decision-dependent distribution programs and they have been analyzed in theory and on a simple numerical example. Future research will look at how to remove some assumptions and use optimistic primal-dual methods, as well as the links between this work and set-valued optimization~\cite{khan2016set} as well as decision-dependent distributionally robust optimization~\cite{luo2020distributionally}. 

\section*{ACKNOWLEDGMENT}

The author thanks Killian Wood and Emiliano Dall'Anese for the careful read of an initial draft of the paper, and for the many useful suggestions for improvements.

\bibliographystyle{ieeetr}
\bibliography{PaperCollection00}

\appendices

\section{Derivation of Problem~\eqref{eq.ex-rob}}\label{ap.1}

We derive~\eqref{eq.ex-rob} following~\cite{duchi2018optimization}. In particular, for the constraints, we compute the worst-case scenario. Start with 
\begin{equation}
D (\A(y) z + b(y)) \leq e, \quad \forall z \in [-1,1]^n.
\end{equation} 
Component-wise, 
\begin{equation}
d_j x - e_j + d_j [{\bf{diag}}(\beta)] z \leq 0  \quad \forall z \in [-1,1]^n,
\end{equation}
so that the worst case is,
\begin{equation}
d_j x - e_j + \|[{\bf{diag}}(d_j)]\beta\|_1 \leq 0.
\end{equation}

As for the other constraint,
\begin{equation}
\|\A(y) z + \bv(y) - x_{\mathrm{ref}}\|^2 \leq \gamma , \quad \forall z \in [-1,1]^n,
\end{equation}
let $s = \A(y) z + \bv(y) - x_{\mathrm{ref}}$, and as such, 
\begin{equation}
s_i = \beta_i z_i + x_i - x_{\mathrm{ref},i},
\end{equation}
and the worst case is,
\begin{equation}
s_i \geq |\beta_i| + |x_i - x_{\mathrm{ref},i}|,
\end{equation}
from which~\eqref{eq.ex-rob}.

\section{Proof of Theorem~\ref{th.1}}\label{ap.2}

Consider the deterministic primal-dual method,
\begin{subequations}\label{pd-d}
\begin{eqnarray}
\check{y}_{k+1} &= &{\mathsf P}_{\mathcal{Y}} \left[ \check{y}_k - \eta \mathbb{E}_{z \in \mathcal{D}(\check{y}_k)}\nabla_{y}\phi (\check{y}_k, \check{\lambda}_k, z) \right], \\
\check{\lambda}_{k+1} &= &{\mathsf P}_{\mathcal{M}} \left[ \check{\lambda}_k + \eta \mathbb{E}_{z \in \mathcal{D}(\check{y}_k)}\nabla_{\lambda}\phi (\check{y}_k, \check{\lambda}_k, z) \right],
\end{eqnarray}
\end{subequations}
which is the deterministic version of~\eqref{pd}, where we have substituted the stochastic gradients with their expectation. Under $\varepsilon L /\mu <1$, for~\cite[Prop.~2.12]{wood2023stochastic}, the fixed point of~\eqref{pd-d} is the unique equilibrium point $(\bar{y}, \bar{\lambda})$. Compactly write~\eqref{pd-d} as, 
\begin{equation}
\check{p}_{k+1} = \mathcal{G}(\check{p}_{k}; \check{p}_k), 
\end{equation}
where we have indicated as $\check{p} = [\check{y}^\top, \check{\lambda}^\top]^\top$, and in the map $\mathcal{G}(\check{p}_{k}; \check{p}_k)$ the second argument represents the dependence of $z$ on $\check{p}_k$. In the same way, we write~\eqref{pd} as,
\begin{equation}
{p}_{k+1} = \tilde{\mathcal{G}}({p}_{k}; {p}_k). 
\end{equation}
Then, for the Triangle inequality,
\begin{multline}
\|p_{k+1} - \bar{p}\| = \|\tilde{\mathcal{G}}(p_{k}; p_k) - \mathcal{G}(\bar{p}; \bar{p})\| \leq \|\mathcal{G}(p_{k}; \bar{p}) - \mathcal{G}(\bar{p}; \bar{p})\| + \\ + \|\mathcal{G}(p_{k}; p_k) - \mathcal{G}(p_k; \bar{p})\| + \|\tilde{\mathcal{G}}(p_{k}; p_k) - \mathcal{G}(p_{k}; p_k)\|.  
\end{multline}

For Assumption~\ref{as.1}-(a) and (b), the gradient map $\psi$ is $\mu$-monotone and $L$-Lipschitz in $(y,\lambda)$ \cite{wood2023stochastic}. As such, we can bound the first right-hand term as,
\begin{equation}
\|\mathcal{G}(p_{k}; \bar{p}) - \mathcal{G}(\bar{p}; \bar{p})\| \leq \sqrt{1 - 2 \eta \mu + \eta^2 L^2} \|p_k - \bar{p}\|.
\end{equation}
For Assumption~\ref{as.1}-(b-c) and~\cite[Lemma 2.9]{wood2023stochastic}, the second right-hand term becomes,
\begin{equation}
 \|\mathcal{G}(p_{k}; p_k) - \mathcal{G}(p_k; \bar{p})\| \leq \eta \varepsilon L \| p_k - \bar{p}\|. 
\end{equation}
Passing now in total expectation, and defining $\varrho =\sqrt{1 - 2 \eta \mu + \eta^2 L^2} +  \eta \varepsilon L$,
\begin{multline}\label{error}
\mathbb{E}[\|p_{k+1} - \bar{p}\|] \leq \varrho \mathbb{E}[\|p_{k} - \bar{p}\|] + \mathbb{E}[\|\tilde{\mathcal{G}}(p_{k}; p_k) - \mathcal{G}(p_{k}; p_k)\|] \leq \\ \underbrace{\leq}_{\eqref{st.setting}} \varrho \mathbb{E}[\|p_{k} - \bar{p}\|] + \eta\sigma.  
\end{multline}
Finally, by iterating~\eqref{error}, the result is proven.

\section{Proof of Theorem~\ref{th.2}}\label{ap.3}

The iterations~\eqref{pd-dm} are a deterministic primal-dual method which is misspecified. We can use the reasoning of Appendix~\ref{ap.2}, to write~\eqref{pd-dm} as
\begin{equation}
{p}_{k+1} = {\mathcal{G}}_{\mathrm{ms}}({p}_{k}; {p}_k). 
\end{equation}
Then, by using the triangle inequality,
\begin{multline}
\|p_{k+1} - \bar{p}\| = \| {\mathcal{G}}_{\mathrm{ms}}(p_{k}; p_k) - \mathcal{G}(\bar{p}; \bar{p})\| \leq \|\mathcal{G}(p_{k}; \bar{p}) - \mathcal{G}(\bar{p}; \bar{p})\| + \\ + \|\mathcal{G}(p_{k}; p_k) - \mathcal{G}(p_k; \bar{p})\| + \| {\mathcal{G}}_{\mathrm{ms}}(p_{k}; p_k) - \mathcal{G}(p_{k}; p_k)\|.  
\end{multline}

The only term left to bound is the rightmost term. We can bound it as, 
\begin{multline}\label{bound1}
\| {\mathcal{G}}_{\mathrm{ms}}(p_{k}; p_k) - \mathcal{G}(p_{k}; p_k)\| \leq \\ \eta\left\|\left[\begin{array}{c} \Ez \nabla_y \phi(y_k, \lambda_k, z) - \Exi \nabla_y \phi(y_k, \lambda_k, z) \\ 
\Ez \nabla_{\lambda} \phi(y_k, \lambda_k, z) - \Exi \nabla_{\lambda} \phi(y_k, \lambda_k, z)
\end{array}\right] \right\|.  
\end{multline}
By using now the same arguments of~\cite[Lemma~2.9]{wood2023stochastic}, namely Kantorovich and Rubinstein duality for the $W_1$ metric, as well as the Lipschitz assumption on the gradient (Cf. Assumption~\ref{as.1}-(b)), we can write,
\begin{equation}
\eqref{bound1} \leq \sqrt{2} \eta L W_1 (\mathcal{D}_{\mathrm{ms}(y)}, \mathcal{D}(y)) \leq \sqrt{2} \eta  L B,     
\end{equation}
where in the last inequality, we have used Assumption~\ref{as.2}. This yields,
\begin{equation}
\|p_{k+1} - \bar{p}\| \leq \varrho \|p_{k} - \bar{p}\| + \sqrt{2} \eta  L B, 
\end{equation}
from which the thesis follows. 

\end{document}